\documentclass[12pt]{article}
\usepackage[dvips]{graphics}
\usepackage{a4,makeidx}
\usepackage{amsfonts}
\usepackage{amsmath}
\usepackage{amssymb}
\usepackage{graphics}
\usepackage{graphicx}
\usepackage{ifthen}
\usepackage{inputenc}
\usepackage{latexsym}
\usepackage{makeidx}
\usepackage{syntonly}
\usepackage{amsthm}
\usepackage{lipsum}

\def\constr#1^#2{\mathrel{\mathop{\kern 0pt#1}\limits^{#2}}}
\def\build#1_#2{\mathrel{\mathop{\kern 0pt#1}\limits_{#2}}}
  
at 8pt

\newtheorem{theorem}{Theorem}[section]
\newtheorem{definition}{Definition}[section]

\newtheorem{corollary}{Corollary}[section]
\newtheorem{proposition}{Proposition}[section]
\newtheorem{lemma}{Lemma}[section]
\newtheorem{remark}{Remark}[section]
\newtheorem{example}{Example}[section]
\numberwithin{equation}{section}

\def\fnt#1#2{\footnotetext{\kern-.3em
    {$^{\mbox{\scriptsize #1}}$}{#2}}}

\chardef\bslchar=`\\ 
\makeatletter
\newcommand{\addbslash}{\expandafter\@addbslash\string}
\def\@addbslash#1{\bslchar\@nobslash#1}
\newcommand{\nobslash}{\expandafter\@nobslash\string}
\def\@nobslash#1{\ifnum`#1=\bslchar\else#1\fi}
\newcommand{\ntt}{\normalfont\ttfamily}
\def\@boxorbreak{\leavevmode
  \ifmmode\hbox\else\ifdim\lastskip=\z@\penalty9999 \fi\fi}
\DeclareRobustCommand{\cs}[1]{\@boxorbreak{\ntt\addbslash#1\@empty}}
\makeatother 

\title{\bf  $(\;\alpha,\beta)$-$A$-NORMAL OPERATORS  IN SEMI-HILBERTIAN SPACES }
\author{ Abdelkader Benali   $^{[1]}$ and  Ould Ahmed Mahmoud Sid Ahmed $^{[2]}$\\
\\
$^{[1]}$ Faculty of science, Mathematics Department,University of
Hassiba\\ Benbouali,
 Chlef Algeria. B.P. 151 Hay Essalem, chlef 02000, Algeria. \\benali4848@gmail.com \\
$^{[2]}$ Mathematics Department, College of Science. Aljouf
University\\Aljouf 2014. Saudi Arabia\\
 sidahmed@ju.edu.sa\\
}
\begin{document}

\maketitle

\begin{abstract}
Let $\mathcal{H}$ be a Hilbert space and let $A$ be a positive bounded operator on $\mathcal{H}$. The semi-inner product $\langle u\;|\;v
\rangle_A:=\langle Au\;|\;v\rangle,\;\;u,v \in \mathcal{H}$  induces a semi-norm $\left\| .\;\right\|_A$ on
$\mathcal{H}.$  This makes $\mathcal{H}$ into a semi-Hilbertian space. In this paper we
introduce and  prove some proprieties of  $(\alpha,\beta)$-normal operators according to semi-Hilbertian
space structures. Furthermore we state various inequalities between the $A$-operator norm and $A$-numerical radius of  $(\alpha,\beta)$-normal   operators  in  semi Hilbertian spaces.\\

\maketitle
\end{abstract}

\begin{center}
{\section{INTRODUCTION AND PRELIMINARIES RESULTS}}
\end{center}
One of the
most important subclasses of the algebra of all bounded linear operators acting
on Hilbert space, the class of normal operators ($TT^*=T^*T$).They have been the object of some intensive studies. The theory
of these operators was investigated in $ [5]$  and $[20]$.\par \vskip 0.2 cm \noindent This class
has been generalized, in some sense, to the larger sets of so-called  quasinormal, hyponormal, isometry, partial isometry, $m$-isometries  operators on Hilbert spaces. \par \vskip 0.2 cm \noindent
Recently, these classes of operators  have been generalized by many authors when an additional semi-inner product is
considered  (see  $[2,3,4,17,18,21]$ ) and other papers.\par \vskip 0.2 cm \noindent In this framework, we show that many results from $[7, 8,12]$ remain true if we
consider an additional semi-inner product defined by a positive semi-definite
operator A. We are interested to introducing a new concept of normality in semi-Hilbertian spaces.\\
The contents of the paper are the following. In Section 1, we give notation
and results about the concept of $A$-adjoint operators that will be useful
in the sequel. In Section 2 we introduce  the  new concept of normality of operators in semi-Hilbertian space $(\mathcal{H},\;\langle .\;|\;.\rangle_A)$, called $(\alpha,\beta)$-$A$-normality   and we
investigate various structural properties of this class of operators.
In Section 3,we state various inequalities between the $A$-operator norm and $A$-numerical radius
of $(\alpha,\beta)$-$A$-normal operators.\par \vskip 0.2 cm \noindent

We start by introducing some notations. The symbol
 $\mathcal{H}$ stands for a complex Hilbert space with inner product $\langle .\;| \;.\;\rangle$ and norm $\|.\|$ . We denote by $\mathcal{B}(\mathcal{H})$  the Banach algebra of all bounded linear operators on $\mathcal{H}$, $I=I_\mathcal{H}$ being  the identity operator. $\mathcal{B}(\mathcal{H})^+$ is the cone of positive (semi-definite) operators, i.e.,\\$\mathcal{B}(\mathcal{H})^+=\{A\in \mathcal{B}(\mathcal{H}) : \langle Au,\;|\;u\rangle\geq 0,\;\forall\;u\in \mathcal{H}\;\}$. For every $T\in \mathcal{L}(\mathcal{H})$  its range is denoted by  $\mathcal{R}(T)$, its null space by $\mathcal{N}(T)$ and its adjoint by $T^*$. If $\mathcal{M}\subset \mathcal{H}$ is a closed subspace, $P_\mathcal{M}$  is the orthogonal projection onto $\mathcal{M}$. The subspace $\mathcal{M}$ is invariant for $T$ if $T\mathcal{M}\subset \mathcal{M}.$ We
shall denote the set of all complex numbers and the complex conjugate of a
complex number $\lambda$ by $ \mathbb{C}$ and $\overline{\lambda}$, respectively.
 The closure of $\mathcal{R}(T)$ will be
denoted by  $\overline{\mathcal{R}(T)}$, and we shall henceforth shorten $ T-\lambda I $ by $ T-\lambda.$
  In addition, if $T,
S \in \mathcal{B}(\mathcal{H})$ then $T \geq S $ means that
$T- S \geq 0.$ .\\
Any $A\in \mathcal{B}(\mathcal{H})^+$ defines a positive semi-definite sesquilinear form, denoted
by
\begin{equation*}
\langle .\;|\;.\rangle_A: \mathcal{B}(\mathcal{H})\times \mathcal{B}(\mathcal{H})\rightarrow \mathbb{C}, \;\langle u\;|\;v \rangle_A=\langle Au \;|\;v\rangle.
 \end{equation*}
 We remark that $\langle u\;|\;v\rangle_A=\langle A^{\frac{1}{2}}u\;|\;A^{\frac{1}{2}}v\rangle.$ The semi-norm induced by $\langle .|.\rangle_A$, which is
denoted by
$\|.\|_A$, is given by $\|u\|_A=\langle u\;|\;u\rangle_A^{\frac{1}{2}}=\| A^{\frac{1}{2}}u\|.$
  This makes $\mathcal{H}$
into a semi-Hilbertian space.
   Observe that $\|u\|_A=0$ if and only if $u\in \mathcal{N}(A).$ Then $\|
.\|_A$ is a norm if and only if $A$ is an injective operator, and the semi-normed space $(\mathcal{B}(\mathcal{H}),\|.\|_A)$ is complete if and only if $\mathcal{R}(A)$ is closed. Moreover $\langle\;|\;\rangle_A $ induced a seminorm on a certain subspace of $\mathcal{B}(\mathcal{H}),$  namely, on the subset of all $T\in \mathcal{B}(\mathcal{H})$ for witch there exists a constant $c>0$ such that $\|Tu\|_A\leq c\|u\|_A$ for every $u\in \mathcal{H}$ ($T$ is called $A$-bounded). For this operators it holds
\begin{equation*}\|T\|_A=\sup_{u\in \overline{\mathcal{R}(A)}\\ ,u\not=0}\frac{\|Tu\|_A}{\|u\|_A}< \infty.
\end{equation*}

It is straightforward that
\begin{equation*}
\|T\|_A = \sup\{| \langle Tu\;|\;v\rangle _A | : u,v \in  \mathcal{H} \;\hbox {and}\;\; \|u\|_A \leq 1, \|v\|_A \leq  1\;\}.
\end{equation*}
\begin{definition} $([2])$
For  $T \in \mathcal{B}(\mathcal{H}),$ an operator $S\in
\mathcal{B}(\mathcal{H})$ is called an $A$-adjoint of $T$ if for
every $ u,v\in \mathcal{H}$
$$\langle Tu\;|\;v\rangle_A=\langle u\;|\;Sv\rangle_A,$$
i.e., $AS=T^*A.$ \par \vskip 0.2 cm \noindent If $T$ is an $A$-adjoint of itself, then $T$ is called an $A$-selfadjoint operator  $ \big( AT=T^*A\big).$
\end{definition}

\noindent It is possible that
an operator $T$ does not have an $A$-adjoint, and if $S$ is an $A$-adjoint of $T$ we may
find many $A$-adjoints; In fact, in $AR=0$ for some $R \in \mathcal{B}(\mathcal{H})$,then $S+R$ is an $A$-adjoint of $T$.
 The set of all $A$-bounded operators which admit an $A$-adjoint is denoted by $\mathcal{B}_A(\mathcal{H}).$
\noindent By Douglas Theorem $( \hbox{see}\;[6,\; 10]\;)$ we have that
$$\mathcal{B}_A(\mathcal{H})=\big\{ \; T \in \mathcal{B}(\mathcal{H}) \;/\; \mathcal{R}(T^*A)\subset \mathcal{R}(A)\;\big\}.$$\noindent If $T\in \mathcal{B}_A(\mathcal{H})$, then there exists a distinguished
$A$-adjoint operator of $T$, namely,the reduced solution of equation $AX=T^*A$, i.e., $A^\dag T^*A$. This operator is denoted by $T^\sharp$. Therefore, $T^\sharp=A^\dag T^*A$ and
$$AT^\sharp=T^*A,\;\mathcal{R}(T^\sharp)\subset \overline{\mathcal{R}(A)}\;\;\hbox{and}\;\;\mathcal{N}(T^\sharp)=\mathcal{N}(T^*A).$$ Note that in which $A^\dag$ is the Moore-Penrose inverse of $A$. For more details see $[2,3,4]$.\par\vskip 0.2 cm \noindent
In the next proposition we collect some properties of $T^\sharp$ and its relationship with the seminorm $\|\;.\;\|_A.$ For the proof see $[2,3 ,4]$ .
\begin{proposition}  Let $T \in \mathcal{B}_A({\mathcal{H}})$. Then the following statements
hold.\par \vskip 0.2 cm
 \noindent(1)\;$T^{\sharp} \in \mathcal{B}_A({\mathcal{H}}),
(T^{\sharp})^{\sharp}=P_{\overline{R(A)}}TP_{\overline{R(A)}}$
and
$(T^{\sharp})^{\sharp})^{\sharp}=T^{\sharp}.$\par
\vskip 0.2 cm \noindent(2)\; If $S\in \mathcal{B}_A(\mathcal{H})$ then $TS \in
\mathcal{B}_A({\mathcal{H}})$ and
$(TS)^{\sharp}=S^{\sharp}
T^{\sharp}.$\par \vskip 0.2 cm
\noindent(3)\; $T^{\sharp} T$ and $TT^{\sharp} $ are
$A$-selfadjoint.\par \vskip 0.2 cm \noindent (4) \;$\|T\|_A =
\|T^{\sharp}\|_A = \|T^{\sharp}
T\|^\frac{1}{2}_ A = \| TT^{\sharp}\|^\frac{1}{2}_ A
$.\par \vskip 0.2 cm \noindent (5)\;$\|S\|_A = \|T^{\sharp}\|_A $
for every $S \in \mathcal{B}({\mathcal{H}}) $ which is an
$A$-adjoint of $T.$\par \vskip 0.2 cm \noindent (6)\;\;If $S \in
\mathcal{B}_A({\mathcal{H}})$ then $\|TS\|_A = \|ST\|_A.$
\end{proposition}
 We recapitulate very briefly the following definitions. For more details, the interested reader is
referred to $[2,4,21]$ and the references therein.

 \begin{definition} Any operator $T\in
\mathcal{B}_Aö(\mathcal{H})$ is called  \par\vskip 0.2 cm \noindent (1)\; $A$-normal if $TT^\sharp=T^\sharp T.$
\par\vskip 0.2 cm \noindent (2)\; $A$-isometry if $T^\sharp T=P_{\overline{\mathcal{R}(A)}}.$\par\vskip 0.2 cm \noindent (3)\; $A$-unitary if  $T^\sharp T=TT^\sharp=P_{\overline{\mathcal{R}(A)}}.$

\end{definition}

\noindent In $[21]$, the $A$-spectral radius  of an operator $T \in \mathcal{B}(\mathcal{H})$, denoted  $r_A(T)$ is defined  as
$$r_A(T) = \limsup_{n\longrightarrow \infty}
\left\|T^n\right\|_A^{\frac{1}{n}}$$
and the $A$-numerical radius of an operator $T \in \mathcal{B}(\mathcal{H})$, denoted by $\omega_A(T)$ is defined as
$$w_A(T) = \sup\big\{\;\left|\langle Tu\;|\;u\rangle_A \right|,\;\;u \in \mathcal{ H}\;, \left\| u \right\|_A = 1\;\;\big\}.$$

It is a generalization of the concept of numerical radius of an operator. Clearly, $\omega_A$
defines a seminorm on $\mathcal{B}(\mathcal{H})$. Furthermore, for every $u \in \mathcal{H},$
$$\left|\left\langle Tu |\;u\right\rangle_A\right| \leq \omega_A(T)\left\|u\right\|_A^2.$$
\par\vskip 0.2 cm \noindent
\begin{remark}
If $T \in \mathcal{B}_A(\mathcal{H})$ is $A$-selfadjoint,then $\left\| T\right\|_A=w_A(T)$ $(\hbox{see}\; [21])$.
\end{remark}
\begin{theorem} $([21],\; \;\hbox{ Theorem 3.1}\;)$\par\vskip 0.2 cm \noindent
A necessary and sufficient condition for an operator $T \in
\mathcal{B}_A(\mathcal{H})$ to be $A$-normal is that \par \vskip 0.2 cm \noindent (1)\;
$\mathcal{R}(TT^\sharp)\subset \overline{\mathcal{R}(A)}$ and \par \vskip 0.2 cm \noindent (2)\;
$\|T^\sharp Tu\|_A=\|TT^\sharp u\|_A$\;\; for all $u \in \mathcal{H}$.
\end{theorem}

\par \vskip 0.2 cm \noindent

\section{ PROPERTIES OF $(\alpha,\beta)$-$A$-NORMAL OPERATORS }
In this section we define the class of $(\alpha,\beta)$-$A$-normal operators according to semi-Hilbertian
space structures and we give some
their proprieties.

 Let $(\alpha,\beta)\in \mathbb{R}^2$ such that  $ 0\leq \alpha \leq 1 \leq  \beta,$ an operator $T \in \mathcal{B}(\mathcal{H})$  is  said to be  $(\alpha,\beta)$-normal $[7,19]$  if $$\alpha^2T^*T\leq TT^*\leq \beta^2T^*T,$$ which is equivalent to the condition
$$\alpha \|Tu\|\leq \|T^* u\|\leq \beta \|Tu\|$$ for all $u \in \mathcal{H}.$ For $\alpha=1=\beta$ is a normal operator. For $\alpha = 1,$ we
observe from the left inequality that $T^*$
is hyponormal and for $\beta= 1,$ from the right
inequality we obtain that $T$ is hyponormal. In recent work,  Senthilkumar $[22]$ introduced $ p$-$(\alpha,\beta)$-normal operators  as a generalization of $(\alpha, \beta)$- normal operators
. An operator $T \in \mathcal{B}(\mathcal{H})$ is said to be $p$-$(\alpha, \beta)$- normal operators for $0 < p \leq 1$
if $$\alpha^2(T^*T)^p \leq  (T T^*)^p \leq \beta^2(T^*T)^p, 0 \leq \alpha \leq 1 \leq \beta .$$ When $p = 1,$
this coincide with $(\alpha, \beta)$-normal operators.\par \vskip 0.2 cm \noindent
Now we are going to consider an extension of the notion of $(\alpha,\beta)$ -normal operators, similar to those
 extensions of the notion of normality to $A$-normality and hyponormality to $A$-hyponormality  (see $[18,21]$).
\begin{definition}$([18])$ Let $A\in \mathcal{B}(\mathcal{H})^+ $ and
 $T\in \mathcal{B}(\mathcal{H})$. We say that $T$  is an  $A$-positive if
$AT \in \mathcal{B}(\mathcal{H})^+$ which is equivalent to the condition
$$\langle Tu\;|\;u \rangle_A \geq  0 \;\;\forall \;u\; \in
\mathcal{H}.$$ We note $T\geq _A0.$
\end{definition}
\begin{definition} $([18])$
An operator  $T \in \mathcal{B}_A(\mathcal{H})$ is said to be
$A$-hyponormal if $T^\sharp T-TT^\sharp$ is $A$-positive  i.e., $T^\sharp T-TT^\sharp \geq_A0.$
\end{definition}
\begin{proposition}$([18])$
Let $T \in \mathcal{B}_A(\mathcal{H})$. Then $T$ is $A$-hyponormal
if and only if $$\|Tu\|_A \geq \|T^\sharp u\|_A\;\;\hbox{for
all}\;\;u \in \mathcal{H}.$$
\end{proposition}
As a generalization of $A$-normal and $A$-hyponormal operators, we introduce $(\alpha,\beta)$-$A$-normal operators.
\begin{definition}
An operator $T\in \mathcal{B}_A(\mathcal{H})$ is said to be $(\alpha,\beta)$-$A$-normal for  $ 0\leq \alpha \leq 1\leq \beta,$ if
$$\beta^2T^\sharp T\geq_A TT^\sharp \geq_A \alpha^2T^\sharp T,$$  which is equivalent to the condition
$$\beta \|Tu\|_A\geq \|T^\sharp u\|_A\geq  \alpha \|Tu\|_A,\;\hbox{for all}\; u \in \mathcal{H}.$$
When $A=I$ (the identity operator),
this coincide with $(\alpha, \beta)$-normal operator.
\begin{itemize} \item For $\alpha=1=\beta$ is a $A$ normal operator.\item For $\beta = 1,$ we
observe from the right inequality that $T$
is $A$-hyponormal.\item For $\alpha= 1,$  and $\mathcal{N}(A)$ is invariant subspace for $T$ from the right
inequality we obtain that $T^\sharp$ is $A$-hyponormal.\end{itemize}
\end{definition}
\begin{remark}
(1) Every $A$-normal operator is $(\alpha,\beta)$-$A$-normal operator.\par\vskip 0.2 cm \noindent (2) If $A$ is injective,then $(1,1)$-$A$-normal is $A$-normal operator. \par\vskip 0.2 cm \noindent (3) If $\mathcal{R}(TT^\sharp) \subset \mathcal{R}(A)$,then $(1,1)$-$A$-normal is $A$-normal operator.
\end{remark}
We  give an example of $(\alpha,\beta)$-$A$-normal operator which is neither $A$-normal nor $A$-hyponormal.
\begin{example}
Let $A=\left(
         \begin{array}{cc}
           1 & 0  \\
          0 & 2  \\
         \end{array}
       \right)
$ and $T=\left(
           \begin{array}{ccc}
             1 & 2  \\
             0& 1 \\
           \end{array}
         \right) \in \mathcal{B}(\mathbb{C}^2).
$ It easy to check that $$A\geq 0,\mathcal{R}(T^*A)\subset \mathcal{R}(A)\hbox{and}\;\;T^\sharp=\left(
                                                                                             \begin{array}{ccc}
                                                                                               1 & 0 \\
                                                                                             1 & 1  \\
                                                                                             \end{array}
                                                                                           \right),T^\sharp T\not=TT^\sharp\hbox{and}\;\;
                                                                                           \|Tu\|_A\not\geq \|T^\sharp u\|_A.
$$
$T$ is neither $A$-normal nor $A$-hyponormal.
Moreover
$$10T^{\sharp}T \geq _A TT^\sharp \geq_A \frac{1}{6} T^\sharp T.
$$
So $T$ is $(\displaystyle\frac{1}{\sqrt{6}},\sqrt{10})$-$A$-normal operator.
\end{example} The following theorem gives a necessary and sufficient conditions that an operator to be  $(\alpha,\beta)$-$A$-normal. It is similar to $[ 14,\;\hbox{Theorem}\; 2.3 ].$
\begin{theorem}
Let $T \in \mathcal{B}_A(\mathcal{H})$ and $(\alpha,\beta)\in \mathbb{R}^2$ such that $0\leq\alpha \leq 1\leq \beta$. Then    $T$ is $(\alpha,\beta)$-$A$-normal  if and only if the following conditions are satisfied
$$
\left\{
\begin{array}{lll}
\lambda ^2TT^\sharp  +2\alpha^2\lambda T^\sharp T + T T^\sharp \geq_A0,\;\;\;\;\hbox{for all}\;\;\lambda \in \mathbb{R} \quad \quad \quad (1)\\
 \\
 \hbox{and}\\
 \\
\lambda ^2T^\sharp T+2\lambda TT^\sharp +\beta^4T^\sharp T \geq _A0,\;\;\;\;\hbox{for all}\;\;\lambda \in \mathbb{R} \quad \quad \quad (2).\
\end{array}
  \right.$$
\end{theorem}
\begin{proof}
Assume that the conditions (1) and (2) are satisfied and prove that $T$ is $(\alpha,\beta)$-$A$-normal.\par \vskip 0.2 cm \noindent
In fact we have  by using elementary properties of real quadratic forms

\begin{eqnarray*}
&& \lambda ^2TT^\sharp +2\alpha^2\lambda T^\sharp T + T T^\sharp \geq_A0\\&\Leftrightarrow& \left\langle (\lambda ^2TT^\sharp +2\alpha^2\lambda T^\sharp T + T T^\sharp) u\;|\;u\right\rangle\geq_A0,\;\; \forall\; u \in \mathcal{H} ,\;\forall\;\lambda \in \mathbb{R}\\&\Leftrightarrow& \lambda ^2\left\|T^\sharp u\right\|_A^2+2\alpha ^2 \lambda \left \|Tu\right\|_A^2+\left\|T^\sharp u\right\|_A^2\geq _A0 ,\;\; \forall\; u \in \mathcal{H},\;\forall\;\lambda \in \mathbb{R}\\
&\Leftrightarrow& \alpha \left\|Tu\right\|_A\leq \left\|T^\sharp u\right\|_A,\;\; \forall\; u \in \mathcal{H}.
\end{eqnarray*}
Similarly
\begin{eqnarray*}
&& \lambda ^2T^\sharp T+2\lambda TT^\sharp +\beta^4T^\sharp T \geq _A0\\&\Leftrightarrow& \left\langle( \lambda ^2T^\sharp T+2\lambda TT^\sharp +\beta^4T^\sharp T ) u\;|\;u\right\rangle\geq_A0,\;\; \forall\; u \in \mathcal{H} ,\;\forall\;\lambda \in \mathbb{R}\\&\Leftrightarrow& \lambda ^2\left\|T u\right\|_A^2+2 \lambda \left \|T^\sharp u\right\|_A^2+ \beta^4\left\|Tu\right\|_A^2\geq _A0 ,\;\; \forall \;u \in \mathcal{H},\;\forall\;\lambda \in \mathbb{R}\\
&\Leftrightarrow& \left\|T^\sharp u\right\|_A\leq \beta \left\|T u\right\|_A,\;\; \forall \; u \in \mathcal{H}.
\end{eqnarray*}

Consequently $$\alpha \left\|Tu\right\|_A\leq \left\|T^\sharp u\right\|_A \leq \beta \left\|T u\right\|_A,\;\; \forall\; u \in \mathcal{H}.$$ So $T$ is $(\alpha,\beta)$-$A$-normal as desired.\par \vskip 0.2 cm \noindent The proof of the converse seems obvious.

\end{proof}

\begin{proposition}
  Let $T\in \mathcal{B}_A(\mathcal{H})$  such that $\mathcal{N}(A)$ is invariant subspace for $T$ and let $(\alpha,\beta) \in \mathbb{R}^2$ such that $0<\alpha \leq 1\leq \beta$ . Then  $T$ is an $(\alpha,\beta)$-$A$-normal if and only if  $T^\sharp$ is
$(\displaystyle\frac{1}{\beta},\frac{1}{\alpha})$-$A$-normal operator.
\end{proposition}
\begin{proof} First assume that
 $T$ is  $(\alpha,\beta)$-$A$-normal operator. We have  for all $u \in \mathcal{H}$
$$\alpha\left\| Tu \right\|_A\leq \left\| T^\sharp u \right\|_A\leq \beta \left\| Tu \right\|_A.$$ It follows that
$$\frac{1}{\beta} \left\| T^\sharp u \right\|_A \leq \left\| T u \right\|_A \;\;\hbox{and}\; \left\| Tu \right\|_A \leq \frac{1}{\alpha} \left\| T^\sharp u \right\|_A.$$ On the other hand,since
 $\mathcal{N}(A)$ is invariant subspace for $T$ we observe that $TP_{\overline{\mathcal{R}(A)}}=P_{\overline{\mathcal{R}(A)}}T$
      and $AP_{\overline{\mathcal{R}(A)}}=P_{\overline{\mathcal{R}(A)}}A=A$ and it follows that

$$\left\| \big(T^\sharp \big)^\sharp u \right\|_A=\left\| P_{\overline{\mathcal{R}(A)}}TP_{\overline{\mathcal{R}(A)}}u \right\|_A =\left\| Tu \right\|_A.$$ Consequently
$$\frac{1}{\beta}\left\| T^\sharp u \right\|_A\leq \left\| \big(T^\sharp \big)^\sharp u \right\|_A\leq \frac{1}{\alpha} \left\| T^\sharp u \right\|_A$$ for all $u\in\mathcal{H}.$ Therefore $T^\sharp$  is $(\displaystyle\frac{1}{\beta},\frac{1}{\alpha})$-$A$-normal operator. \par \vskip 0.2 cm \noindent Conversely assume that $T^\sharp$ is $(\displaystyle\frac{1}{\beta},\frac{1}{\alpha})$-$A$-normal operator. We have
$$\frac{1}{\beta}\left\| T^\sharp u \right\|_A\leq \left\| \big(T^\sharp \big)^\sharp u \right\|_A\leq \frac{1}{\alpha} \left\| T^\sharp u \right\|_A$$ for all $u\in\mathcal{H},$ and from which it follows that
$$\frac{1}{\beta}\left\| T^\sharp u \right\|_A\leq \left\|T u \right\|_A\leq \frac{1}{\alpha} \left\| T^\sharp u \right\|_A$$ for all $u\in\mathcal{H}.$ Hence
$$\alpha\left\| Tu \right\|_A\leq \left\| T^\sharp u \right\|_A\leq \beta \left\| Tu \right\|_A,\;\forall\; u \in \mathcal{H}.$$ This completes the proof.
\end{proof}
The following corollary is a immediate consequence of Proposition 2.2.
\begin{corollary}
  Let $T\in \mathcal{B}_A(\mathcal{H})$  such that $\mathcal{N}(A)$ is invariant subspace for $T$ and let $(\alpha,\beta) \in \mathbb{R}^2$ such that $0<\alpha \leq 1\leq \beta$  and $\alpha \beta=1$. Then  $T$ is an $(\alpha,\beta)$-$A$-normal if and only if  $T^\sharp$ is
$(\alpha,\beta)$-$A$-normal operator.
\end{corollary}

\begin{remark}
$(\alpha,\beta)$-$A$-normality is not translation invariant, more precisely, there exists an operator $T\in \mathcal{B}_A(\mathcal{H})$ that $T$ is $(\alpha,\beta)$-$A$-normal,but $T+\lambda$ is not $(\alpha,\beta)$-$A$-normal for some $\lambda \in \mathbb{C}.$ The following example shows
that such operators exist:
\end{remark}
\begin{example} Consider the operators
 $A=\left(
         \begin{array}{cc}
           1 & 0  \\
          0 & 2  \\
         \end{array}
       \right)
$ ,$T=\left(
           \begin{array}{ccc}
             1 & 2  \\
             0& 1 \\
           \end{array}
         \right)$ and $S=T+I=\left(
           \begin{array}{ccc}
             2 & 2  \\
             0& 2 \\
           \end{array}
         \right) \in \mathcal{B}(\mathbb{R}^2)$. It is easily to check that $T$ is $(\displaystyle\frac{1}{\sqrt{6}},\sqrt{10})$-$A$-normal, but $S$ is not $(\displaystyle\frac{1}{\sqrt{6}},\sqrt{10})$-$A$-normal. So $(\alpha,\beta)$-$A$-normality is not translation-invariant.
\end{example}
Similarly to $[12]$, we define the following quantities
$$\mu_A^1(T)=\inf\bigg\{ \;\frac{Re\left\langle Tu\;|\;u\right\rangle_A}{\|Tu\|_A},\;\|u\|_A=1,\;\;Tu \notin \mathcal{N}(A^{\frac{1}{2}})\;\bigg\} $$ and
$$\mu_A^2(T)=\sup\bigg\{ \;\frac{Re\left\langle Tu\;|\;u\right\rangle_A}{\|Tu\|_A},\;\|u\|_A=1,\;\;Tu \notin \mathcal{N}(A^{\frac{1}{2}})\;\bigg\}. $$
\noindent A.Saddi $[ 21,\;\hbox{Corollary}\;3.2\;]$ have shown that if $T$ is $A$-normal operator such that $\mathcal{N}(A)$ is invariant subspace for $T$, then $T-\lambda$ is $A$-normal. In $[18,\;\hbox{Theorem}\;2.7\;]$ the authors proved this property for $A$-hyponormal operators. In the following theorem we extend these results to
$(\alpha,\beta)$-$A$-normal operators. This is a generalization of $ [12,\;\hbox{ Theorem}\; 2.1 ].$
      \begin{theorem}
      Let $T\in \mathcal{B}_A(\mathcal{H})$ such that $\mathcal{N}(A)$ is invariant subspace for $T$  and $0\leq \alpha \leq 1 \leq \beta$.  The following statements hold. \par \vskip 0.2 cm \noindent (1)\; If $T$ is $(\alpha,\beta)$-$A$-normal, then $\lambda T$ is $(\alpha,\beta)$-$A$-normal for $\lambda \in \mathbb{C}$ .\par \vskip 0.2 cm \noindent (2)\;If $T$ is $(\alpha,\beta)$-$A$-normal,
        then $T+\lambda$ for $\lambda \in \mathbb{C}$ is $(\alpha,\beta)$-$A$-normal, if one of the following conditions holds:\par\vskip 0.2cm \noindent (i)\quad $\mu_A^1(\overline{\lambda}T)\geq 0$
      \par\vskip 0.2cm \noindent (ii)\quad $\mu_A^1(\overline{\lambda}T)<0,$ $|\lambda|^2+2|\lambda|\|T\|_A\mu_A^1(\overline{\lambda}T) > 0.$
      \end{theorem}
      \begin{proof}  (1) Since  $\mathcal{N}(A)$ is invariant subspace for $T$ we observe that $TP_{\overline{\mathcal{R}(A)}}=P_{\overline{\mathcal{R}(A)}}T$
      and $AP_{\overline{\mathcal{R}(A)}}=P_{\overline{\mathcal{R}(A)}}A=A.$ Let $T$ be $(\alpha,\beta)$-$A$-normal then
      \begin{eqnarray*}
      \beta^2T^\sharp T \geq_ATT^\sharp \geq_A\alpha^2T^\sharp T&\Leftrightarrow& \beta^2 |\lambda|^2T^\sharp T \geq_A |\lambda|^2TT^\sharp \geq_A |\lambda|^2\alpha^2T^\sharp T\\&\Leftrightarrow&  A\overline{\lambda} T^\sharp \lambda T\geq A\lambda T \overline{\lambda}T^\sharp \geq  \alpha^2A\overline{\lambda}T^\sharp \lambda T \\&\Leftrightarrow& AP_{\overline{\mathcal{R}(A)}}\overline{\lambda} T^\sharp \lambda T\geq AP_{\overline{\mathcal{R}(A)}}\lambda T \overline{\lambda}T^\sharp \geq  \alpha^2AP_{\overline{\mathcal{R}(A)}}\overline{\lambda}T^\sharp \lambda T
      \\&\Leftrightarrow& \beta^2A(\lambda T)^\sharp (\lambda T)\geq A(\lambda T)(\lambda T)^\sharp \geq \alpha^2 A(\lambda T)^\sharp (\lambda T)
      \\&\Leftrightarrow& \beta^2(\lambda T)^\sharp (\lambda T)\geq_ A(\lambda T)(\lambda T)^\sharp \geq_A \alpha^2 (\lambda T)^\sharp (\lambda T).
      \end{eqnarray*}
Therefore $\lambda T $ is $(\alpha,\beta)$-$A$-normal operator.\par \vskip 0.2 cm \noindent

  \noindent (2)    Assume that $T$ is $(\alpha,\beta)$-$A$-normal and  the condition $(i)$ holds.
We need to prove that
\begin{eqnarray} \left\{\begin{array}{c}
                         \alpha^2\left\langle \big(T+\lambda\big)^\sharp\big(T+\lambda\big)u\;|\;u \right\rangle_A\leq \left\langle \big(T+\lambda\big)\big(T+\lambda\big)^\sharp u\;|\;u \right\rangle_A \\
                          \\
                          \left\langle \big(T+\lambda\big)\big(T+\lambda\big)^\sharp u\;|\;u \right\rangle_A\leq \beta^2\left\langle \big(T+\lambda\big)^\sharp\big(T+\lambda\big) u\;|\;u \right\rangle_A.
                        \end{array}\right.
\end{eqnarray}
In order To verify $(2.1)$ we have
\begin{eqnarray*}
\alpha^2\left\langle \big(T+\lambda\big)^\sharp\big(T+\lambda\big)u\;|\;u \right\rangle_A&=& \alpha^2\bigg\{
\left\langle T^\sharp Tu\;|u \right\rangle_A+\left\langle {\lambda}T^\sharp u\;|\;u \right\rangle_A +\left\langle\overline{\lambda}P_{\overline{\mathcal{R}(A)}}Tu\; |\;u \right\rangle_A \\&&\quad \quad+|\lambda|^2\left\langle P_{\overline{\mathcal{R}(A)}}u\; |\;u \right\rangle_A   \bigg\}\\&=&
\alpha^2\bigg\{
\left\langle T^\sharp Tu\;|u \right\rangle_A+ 2Re\left\langle\overline{\lambda}Tu\; |\;u \right\rangle_A +|\lambda|^2\left\|u \right\|_A ^2  \bigg\}\\&\leq& \alpha^2
\left\langle T^\sharp Tu\;|u \right\rangle_A+ \alpha^2\bigg\{2Re\left\langle\overline{\lambda}Tu\; |\;u \right\rangle_A +|\lambda|^2\left\|u \right\|_A ^2  \bigg\}.
\end{eqnarray*}
The condition $(i)$ implies that $ 2Re\left\langle\overline{\lambda}Tu\; |\;u \right\rangle_A\geq 0$ and it follows that
\begin{eqnarray*}
\alpha^2\left\langle \big(T+\lambda\big)^\sharp\big(T+\lambda\big)u\;|\;u \right\rangle_A&\leq&
\bigg\{
\left\langle TT^\sharp u\;|u \right\rangle_A+2Re\left\langle\overline{\lambda}Tu\; |\;u \right\rangle_A +|\lambda|^2\left\|u \right\|_A ^2  \bigg\}\\&=& \left\langle \big(T+\lambda\big)\big(T+\lambda\big)^\sharp u\;|\;u \right\rangle_A
\\&=&
\bigg\{\left\langle TT^\sharp u\;|u \right\rangle_A+2Re\left\langle\overline{\lambda}Tu\; |\;u \right\rangle_A +|\lambda|^2\left\|u \right\|_A ^2  \bigg\}
\\&\leq&\beta^2\left\langle \big(T+\lambda\big)^\sharp\big(T+\lambda\big)u\;|\;u \right\rangle_A
 \end{eqnarray*}
 and hence $T+\lambda$ is $(\alpha,\beta)$-$A$-normal.
On the other hand if the condition (ii) is satisfied then we have for $\lambda \not=0$
\begin{eqnarray*}
&&|\lambda|^2+2|\lambda| \left\| T\right\|_A\mu_A^1(\overline{\lambda}T)\\&=&|\lambda|^2+2|\lambda|\bigg(\sup_{\|u\|_A=1}\|Tu\|_A\bigg)\bigg(\inf\bigg\{ \;\frac{Re\left\langle\overline{\lambda} Tu\;|\;u\right\rangle_A}{|\lambda|\|Tu\|_A},\;\|u\|_A=1,\;\;Tu \notin \mathcal{N}(A^{\frac{1}{2}})\;\bigg\} \bigg)\\&\leq&|\lambda|^2+2\inf_{\|u\|_A=1}Re\left\langle\overline{\lambda} Tu\;|\;u\right\rangle_A\\&\leq&|\lambda|^2+2Re\left\langle\overline{\lambda} Tu\;|\;u\right\rangle_A.
\end{eqnarray*}
öA similar argument  used  as above shows that $T+\lambda$ is $(\alpha,\beta)$-$A$-normal.

      \end{proof}
\begin{corollary} Let $T \in \mathcal{B}_A(\mathcal{H})$ be an $(\alpha,\beta)$-$A$-normal operator. The following statement hold
\par \vskip 0.2 cm \noindent (1)\;If $\mu_A^1(T)\geq 0$ then $T+\lambda$ is $(\alpha,\beta)$-$A$-normal for every $\lambda >0.$ \par \vskip 0.2 cm \noindent (1)\;If $\mu_A^2(T)\leq 0$ then $T+\lambda$ is $(\alpha,\beta)$-$A$-normal for every $\lambda <0.$
\end{corollary}

\begin{proof}
\par \vskip 0.2 cm \noindent (1)\;For every $\lambda >0$ we have $\mu_A^1(\overline{\lambda}T)=\mu_A^1(\lambda T)=\mu_A^1(T)\geq 0$.By using Theorem 2.2 (i) we have that $T +\lambda$ is an $(\alpha,\beta)$-$A$-normal.
\par \vskip 0.2 cm \noindent (2)\;For every $\lambda <0$ we have $\mu_A^1(\overline{\lambda}T)=-\mu_A^2( T)\geq 0$.By using Theorem 2.2 (ii) we have that $T+\lambda $ is an $(\alpha,\beta)$-$A$-normal.
\end{proof}

\begin{lemma} $( [18], \;\hbox{Lemma}\; 2.1\;)$
Let $T,S \in \mathcal{B}(\mathcal{H})$  such that $T \geq _A S$ and
let $R \in \mathcal{B}_A(\mathcal{H}). $ Then the following properties
hold
\par \vskip 0.2 cm \noindent (1)\; $R^\sharp TR\geq _A R^\sharp SR.$ \par
\vskip 0.2 cm \noindent (2)\; $R TR^\sharp\geq _A R SR^\sharp.$
\par
\vskip 0.2 cm \noindent (3)\;If $R$ is $A$-selfadjoint  then $RTR\geq_ARSR.$
\end{lemma}
\begin{proposition}
Let $T,V \in \mathcal{B}_A(\mathcal{H})$ such that $\mathcal{N}(A)$ is invariant  subspace for both $T$ and $V$.If $T$ is an $(\alpha,\beta)$-$A$-normal  $(0\leq \alpha \leq 1\leq \beta )$ and $V$ is an $A$-isometry , then $VTV^\sharp$ is an $(\alpha,\beta)$-$A$-normal operator.
\end{proposition}
\begin{proof}
Assume that $\beta^2 T^\sharp T \geq_ATT^\sharp\geq_A \alpha^2T^\sharp V$ and $V^\sharp V= P_{\overline{\mathcal{R}(A)}}.$ This implies
\begin{eqnarray*}
\beta^2\big(VTV^\sharp \big)^\sharp \big(VTV^\sharp \big)&=& \beta^2\bigg(\big(V^\sharp)^\sharp T^\sharp V^\sharp VTV^\sharp \bigg)\\&=&\beta^2\bigg(P_{\overline{\mathcal{R}(A)}}VP_{\overline{\mathcal{R}(A)}}T^\sharp P_{\overline{\mathcal{R}(A)}}TV^\sharp \bigg)\\&=&\beta^2\bigg( VP_{\overline{\mathcal{R}(A)}} T^\sharp T \big( VP_{\overline{\mathcal{R}(A)}}\big)^\sharp\bigg)\\&\geq_A& VP_{\overline{\mathcal{R}(A)}} T T^\sharp \big( VP_{\overline{\mathcal{R}(A)}}\big)^\sharp\;\quad \quad(\hbox{by Lemma 2.1})\\&\geq_A& \big(VTV^\sharp\big)\big(VTV^\sharp \big)^\sharp.
\end{eqnarray*}
Similarly, we have
\begin{eqnarray*}
    \big(VTV^\sharp\big)\big(VTV^\sharp \big)^\sharp   &=&  VP_{\overline{\mathcal{R}(A)}} T T^\sharp \big( VP_{\overline{\mathcal{R}(A)}}\big)^\sharp\\&\geq_A& \alpha^2VP_{\overline{\mathcal{R}(A)}}  T^\sharp T \big( VP_{\overline{\mathcal{R}(A)}}\big)^\sharp\;\quad \quad(\hbox{by Lemma 2.1})\\&\geq_A&
    \alpha^2\big(VTV^\sharp\big)^\sharp\big(VTV^\sharp\big).
\end{eqnarray*}
The conclusion holds.
\end{proof}
\begin{proposition}
Let  $T,S \in \mathcal{B}_A(\mathcal{H})$ such that $T$ is $(\alpha,\beta)$-$A$-normal and $S$ is $A$-selfadjoint. If $T^\sharp S=ST^\sharp$ then $TS$ is $(\alpha,\beta)$-$A$-normal.
\end{proposition}
\begin{proof}
Since $T$ is $(\alpha,\beta)$-$A$-normal we have for $u \in \mathcal{H}$
$$\alpha \left\| TSu\right\|_A\leq \left\| T^\sharp Su\right\|_A\leq \beta \left\| TSu\right\|_A.$$
On the other hand
\begin{eqnarray*}
\left\| T^\sharp Su\right\|_A^2= \left\langle T^\sharp Su\;|\;T^\sharp Su\right\rangle_A=\left\langle AST^\sharp u\;|\;ST^\sharp u\right\rangle = \left\langle (TS)^\sharp u\;|\;(TS)^\sharp u\right\rangle_A=\left\| (TS)^\sharp u\right\|_A^2.
\end{eqnarray*}
This implies
$$\alpha \left\| TSu\right\|_A\leq \left\| (TS)^\sharp u\right\|_A\leq \beta \left\| TSu\right\|_A.$$
\end{proof}

\begin{proposition}
Let  $T,S \in \mathcal{B}_A(\mathcal{H})$ such that $T$ is $(\alpha,\beta)$-$A$-normal and  $S$ is $A$-unitary. If $T S=ST$  and $\mathcal{N}(A)$ is invariant subspace for $T$  then $TS$ is $(\alpha,\beta)$-$A$-normal.
\end{proposition}
\begin{proof} Since $\mathcal{N}(A)$ is invariant subspace for $T$ we observe that $TP_{\overline{\mathcal{R}(A)}}=P_{\overline{\mathcal{R}(A)}}T$ and $T^\sharp P_{\overline{\mathcal{R}(A)}}=P_{\overline{\mathcal{R}(A)}}T^\sharp.$ Let $S$ be $A$-unitary then $S^\sharp S=SS^\sharp =P_{\overline{\mathcal{R}(A)}}.$ \par \vskip 0.2 cm \noindent Now it is easy to see that
$$\beta^2\bigg((TS)^\sharp(TS)\bigg)=\beta^2\bigg(T^\sharp S^\sharp ST\bigg)=\beta^2\bigg(T^\sharp P_{\overline{\mathcal{R}(A)}}T\bigg)=\beta^2\bigg(P_{\overline{\mathcal{R}(A)}}T^\sharp T P_{\overline{\mathcal{R}(A)}}\bigg).$$
By using the fact that $T$ is $(\alpha,\beta)$-$A$-normal,it follows immediately from  Lemma 2.1 that
$$\beta^2\bigg((TS)^\sharp(TS)\bigg)\geq_A\underbrace{\bigg(P_{\overline{\mathcal{R}(A)}}T T^\sharp P_{\overline{\mathcal{R}(A)}}\bigg)}_{(1)}\geq_A\alpha^2\underbrace{(P_{\overline{\mathcal{R}(A)}}T^\sharp T P_{\overline{\mathcal{R}(A)}}\bigg)}_{(2)}.$$

Notice that
 (1) gives
$$(P_{\overline{\mathcal{R}(A)}}T T^\sharp P_{\overline{\mathcal{R}(A)}}=T P_{\overline{\mathcal{R}(A)}}T^\sharp= TSS^\sharp T^\sharp=TS(TS)^\sharp$$
 and similarly (2) gives
$$(P_{\overline{\mathcal{R}(A)}}T^\sharp T P_{\overline{\mathcal{R}(A)}}=T^\sharp P_{\overline{\mathcal{R}(A)}}T= T^\sharp S^\sharp ST=(TS)^\sharp (TS).$$ So
$$\beta^2(TS)^\sharp(TS) \geq_A TS(TS)^\sharp \geq_A \alpha^2(TS)^\sharp (TS).$$
Hence $TS$ is $(\alpha,\beta)$-$A$-normal operator.
\end{proof}
The  following  example  proves  that  even  if  $T$ and  $S$  are  $(\alpha,\beta)$-$A$--normal operators,  their  product  $TS$  is  not  in general $(\alpha,\beta)$-$A$-normal operator.
\begin{example}
\noindent (1) Consider $T=\left(
                            \begin{array}{ccc}
                              0 & 0 & 1 \\
                              0 & 1& 0\\
                              1 & 0 &0 \\
                            \end{array}
                          \right)$ and $S=\left(
                            \begin{array}{ccc}
                              -1 & 0 & 0 \\
                              0 & -1& 0\\
                              0 & 0 &-1 \\
                            \end{array}
                          \right)
$ which are $(\alpha,\beta)$-$I_3$-normal and their product is $(\alpha,\beta)$-$I_3$-normal.\par \vskip 0.2 cm
\noindent (2)\; Consider $T=\left(
                              \begin{array}{cc}
                                1 & 0 \\
                                1 & 1 \\
                              \end{array}
                            \right)
$ and $S=\left(
                              \begin{array}{cc}
                                -1 & 0 \\
                                0 & -1 \\
                              \end{array}
                            \right)$ which are
  $(\alpha,\beta)$-$I_2$-normal whereas their product $TS=\left(
                              \begin{array}{cc}
                                -1 & 0 \\
                                -1 & -1 \\
                              \end{array}
                            \right)$ is not $(\alpha,\beta)$-$I_2$-normal.
\end{example}
\begin{theorem}
Let $T,S \in \mathcal{B}_A(\mathcal{H})$ such that $T$ is $(\alpha,\beta)$-$A$-normal $(0\leq\alpha \leq 1\leq \beta)$   and $S$ is $(\alpha^\prime,\beta^\prime ) $ -$A$-normal $( 0\leq\alpha^\prime \leq 1\leq \beta^\prime).$ Then the following statements hold:\par \vskip 0.2 cm \noindent (1)
If $T^\sharp S=ST^\sharp$,then $TS$ is $(\alpha \alpha^\prime,\beta \beta^\prime)$-$A$-normal operator.
\par \vskip 0.2 cm \noindent (2)
If $S^\sharp T=TS^\sharp$,then $ST$ is $(\alpha \alpha^\prime,\beta \beta^\prime)$-$A$-normal operator.
\end{theorem}
\begin{proof}(1)\;Since  $T$ is $(\alpha,\beta)$-$A$-normal and $S$ is $(\alpha^\prime,\beta^\prime )$- $A$-normal, it follows that
for all $u \in \mathcal{H}$
\begin{eqnarray*}
\alpha \alpha^\prime\left\| TSu \right\|_A\leq \alpha^\prime \left\|T^\sharp Su \right\|_A&=&\alpha^\prime\left\| ST^\sharp u \right\|_A\\&\leq&
\left\| S^\sharp T^\sharp u \right\|_A\\&=&\left\|(TS)^\sharp u  \right\|_A\\&\leq&\beta^\prime \left\| ST^\sharp u \right\|_A \\&=&\beta^\prime \left\| T^\sharp S u \right\|_A \\&\leq&\beta\beta^\prime\left\| TS u \right\|_A .
\end{eqnarray*}
It follows that $$\alpha\alpha^\prime\left\| S Tu \right\|_A\leq \left\| (TS)^\sharp  u \right\|_A\leq \beta \beta^\prime\left\| TS u \right\|_A .$$
The proof of the second assertion is completed in much the same way as the first
assertion.
\end{proof}

The following example shows that the power of $(\alpha,\beta)$-$A$-normal operator not necessarily an $(\alpha,\beta)$-$A$-normal.
\begin{example}
Let $A=\left(
         \begin{array}{cc}
           1 & 0  \\
          0 & 2  \\
         \end{array}
       \right)
$ and $T=\left(
           \begin{array}{ccc}
             1 & 2  \\
             0& 1 \\
           \end{array}
         \right) \in \mathcal{B}(\mathbb{C}^2).$ By Example 2.1,$T$ is $(\displaystyle\frac{1}{\sqrt{6}}, \sqrt{10})$-$A$-normal.However by
direct computation one can show that $T^2$ is is neither $(\displaystyle\frac{1}{\sqrt{6}}, \sqrt{10})$-$A$-normal  nor $(\displaystyle\frac{1}{6},10)$-$A$-normal. But it is $(\displaystyle\frac{1}{36}, 100)$-$A$-normal i.e., $T^2$ is $\big( \big(\displaystyle\frac{1}{\sqrt{6}}\big)^{2^2},(\sqrt{10})^{2^2}\big)$-$A$-normal.
         \end{example}
\noindent{\bf Question.}\;
          If $T\in \mathcal{B}_A(\mathcal{H})$ which is   $(\alpha,\beta)$-$A$-normal operator, is that true  $T^n$ is  $(\alpha^{n^2},\beta^{n^2})$-$A$-normal operator?
\begin{remark}
Let $T\in \mathcal{B}_A(\mathcal{H})$ ,then\par \vskip 0.2 cm \noindent (1)\;If $T$ is $A$-normal, then $r_A(T)=\left\|T\right\|_A$ \;(see,$[21, \hbox{Corollary}\;3.2]$).
\par \vskip 0.2 cm \noindent (2)\;If $T$ is $A$-hyponormal, then $r_A(T)=\left\|T\right\|_A$ \;(see,$[ 18, \hbox{Theorem}\;2.6]$).
\end{remark}
\noindent
The following theorem presents a generalization
of these results to $(\alpha,\beta)$-$A$-normal. Our inspiration cames from $[12,\;\hbox{Theorem}\;2.5]$.
\begin{theorem}
 Let $T\in \mathcal{B}_A(\mathcal{H})$   be an $(\alpha,\beta)$-$A$-normal such that $T^{2^{n}}$ is $(\alpha,\beta)$-$A$-normal for every $n\in \mathbb{N}$,too. Then, we have
$$\frac{1}{\beta}\left\| T\right\| _A\leq r_A(T)\leq \left\| T\right\|_A.$$
\end{theorem}
\begin{proof}
It is we know that if  $T\in \mathcal{B}_A(\mathcal{H})$ then $$\left\| T^\sharp T\right\|_A=\left\|TT^\sharp \right\|_A= \left\| T\right\|_A^2$$ and if $T$ is $A$-selfadjoint then $$\left\|T^2 \right\|_A=\left\| T\right\|_A^2.$$
From the definition of $(\alpha,\beta)$-$A$-normal operator and Lemma 2.1.1 we deduce that
$$\beta^2\big(T^\sharp\big)^2T^2\geq_A\big(T^\sharp T\big)^2\geq_A \alpha^2\big(T^\sharp\big)^2T^2$$ and so
$$\sup_{\|u\|_A=1}\left\langle \big(T^\sharp\big)^2T^2u\;|\;u\right\rangle_A \geq \frac{1}{\beta ^2}\sup_{\|u\|_A=1}\left\langle \big(T^\sharp T\big)^2u\;|\;u\right\rangle_A .$$ Hence
$$\left\| \big(T^\sharp\big)^2T^2\right\|_A^2 \geq \frac{1}{\beta^2} \left\| \big(T^\sharp T\big)^2\right\|_A^2 =\frac{1}{\beta^2}\left\| T\right\|_A^4.$$
Now using a mathematical induction, we observe that for every positive integer
number $n$,
$$\left \| \big(T^\sharp\big)^{2^{n}}T^{2^{n}}\right\|_A\geq \frac{1}{\beta^{2^{n+1}-2}}\left\|T \right\|_A^{2^{n+1}}.$$
We have
\begin{eqnarray*}
r_A(T)^2=r_A(T^\sharp)r_A(T)&=&\lim\sup_{n \longrightarrow \infty} \left\| \big(T^\sharp\big)^{2^{n}}\right\|_A^{\frac{1}{2^n}}\lim\sup_{n \longrightarrow \infty}\left\| T^{2^{n}}\right\|_A^{\frac{1}{2^n}}\\&\geq&\lim_{n\longrightarrow \infty}
\bigg(\left\| \big(T^\sharp\big)^{2^{n}}\right\|_A\left\| T^{2^{n}}\right\|_A \bigg)^{\frac{1}{2^n}}
\\&\geq&\lim_{n\longrightarrow \infty}
\bigg(\left\| \big(T^\sharp\big)^{2^{n}} T^{2^{n}}\right\|_A \bigg)^{\frac{1}{2^n}}\\&\geq& \frac{1}{\beta^2}\left\|T\right\|_A^2 .
\end{eqnarray*}
Therefore, we get $$\frac{1}{\beta}\left\|T\right\|_A \leq r_A(T)\leq \left\|T\right\|_A.$$ This completes the proof.
\end{proof}
 Let $\mathcal{H} \overline{\otimes} \mathcal{H}$ denote the
 completion, endowed with a reasonable uniform crose-norm, of the
 algebraic tensor product $\mathcal{H} {\otimes} \mathcal{H}$ of
 $\mathcal{H}$ with $\mathcal{H}$.
 Given non-zero $T , S \in \mathcal{B}(\mathcal{H}),$ let $T \otimes S \in \mathcal{B}(\mathcal{H} \overline{\otimes} \mathcal{H})$ denote the tensor product on the Hilbert space
$\mathcal{H} \overline{\otimes} \mathcal{H}$, when $T\otimes S$ is
defined as follows $$\langle T\otimes S(\xi_1\otimes
\eta_1)|\;(\xi_2\otimes \eta_2)\rangle= \langle T\xi_1 |\;
\xi_2\rangle \langle S\eta_1|\;\eta_2\rangle.$$ The operation of
taking tensor products $T \otimes S$ preserves many properties of $T
, S \in \mathcal{B}(\mathcal{H})$, but by no means all of them.
Thus,
 whereas $T\otimes S$ is normal if and only if $T$  and $S$ are normal $[15]$,
there exist paranormal operators $T$ and $S$ such that $T\otimes S$
is not paranormal $[1]$. In $[9]$, Duggal showed that if for non-zero
$T , S \in \mathcal{B}(\mathcal{H}), T\otimes S$ is $p$-hyponormal
if and only if $T$ and $S$ are $p$-hyponormal. Thus result was
extended to $p$-quasi-hyponormal operators in $[16]$
.\par \vskip 0.2 cm
\noindent
Recall that for $T \in \mathcal{B}_A(\mathcal{H})$  and  $S \in \mathcal{B}_B(\mathcal{H})$,\; $T\otimes S$ is $(\alpha,\beta)$-$(A\otimes B)$-normal operator with $0\leq \alpha \leq 1\leq \beta$, if
$$\beta^2\big(T\otimes S\big)^\sharp \big(T\otimes S\big)\geq _{A\otimes B}\big(T\otimes S\big)\big(T\otimes S\big)^\sharp\geq_{A\otimes B}\alpha^2\big(T\otimes S\big)^\sharp \big(T\otimes S\big)$$ or equivalently
$$\alpha \left\|\big(T\otimes S\big)\big(u\otimes v\big) \right\|_{A\otimes B}\leq \left\| \big(T\otimes S\big)^\sharp\big(u\otimes v\big)\right\|_{A\otimes B}\leq \beta \left\|\big(T\otimes S\big)\big(u\otimes v\big) \right\|_{A\otimes B},$$
for all $u,v \in \mathcal{H}.$
\begin{lemma} $(\;[18],\;\hbox{Lemma}\;3.1\;)$ \\Let $T_k,S_k \in \mathcal{B}(\mathcal{H}),\;k=1,2$ and Let $A,B \in  \mathcal{B}(\mathcal{H})^+, $
such that  $T_1 \geq _{A} T_2 \geq_{A} 0$ and $S_1 \geq_{B}
S_2\geq_{B} 0$, then $$\big(T_1\otimes S_1\big) \geq_{A \otimes
B}\big( T_2\otimes S_2\big) \geq_{A \otimes B} 0.$$

\end{lemma}
\begin{proposition} $(\;[18],\;\hbox{Proposition}\;3.2\;)$ Let $T_1,T_2,S_1,S_2\in \mathcal{B}(\mathcal{H})$  and let $A,B \in \mathcal{B}(\mathcal{H})^+$ such that $T_k$
is $A$- positive and $S_k$ is $B$-positive for $k=1,2$. If
$T_1\not=0$ and $S_1\not=0$,then the following conditions are
equivalents   \par \vskip 0.2 cm \noindent (1)\; $  T_2\otimes S_2
\geq _{A\otimes B} T_1\otimes S_1$
\par \vskip 0.2 cm \noindent (2)\; there exists $d >0$ such that
$dT_2 \geq _A T_1$ and $ d^{-1}S_2 \geq_B S_1.$
\end{proposition}
The following theorem gives a necessary and sufficient condition for $T \otimes S$ to be  $(\alpha,\beta)$-
  -$A\otimes
 B$-normal operator when $T$ and $S$ are both nonzero operators.
\begin{proposition}
Let $T \in \mathcal{B}_A(\mathcal{H})$  and let $S \in \mathcal{B}_B(\mathcal{H})$ with $T\neq 0$ and $S\neq 0.$ Let $(\alpha,\beta) \in \mathbb{R}^2$ and $(\alpha^\prime,\beta^\prime) \in \mathbb{R}^2$ such that $0\leq \alpha,\;\alpha^\prime \leq 1$ and $1\leq \beta,\;1 \leq \beta^\prime.$ The following properties hold:\par \vskip 0.2 cm \noindent (1)\; If
$T$ is  an $(\alpha,\beta)$-$A$ normal and $S$ is an $(\alpha^\prime,\beta^\prime)$-$B$-normal ,then $T\otimes S$
is a $(\alpha \alpha^\prime ,\beta\beta^\prime)$-$A\otimes B$-normal operator. \par \vskip 0.2 cm \noindent (2)\;If $T\otimes S$ is an $( \alpha,\beta)$-$A\otimes B$-normal, then there exist two constants $d>0$ and $d_0>0$ such
that $T$ is $\displaystyle \big(\sqrt{d_0^{-1}} \alpha,\sqrt{{d}} \beta\big)$-$A$-normal and $S$ is
 $ \big(\sqrt{d_0},\displaystyle\frac{1}{\sqrt{d}}\big)$-$B$-normal operator.
\end{proposition}
\begin{proof} Assume that $T$ is  an $(\alpha,\beta)$-$A$ normal and $S$ is an $(\alpha^\prime,\beta^\prime)$-$B$-normal.
By assumptions we have $$
\beta^2 T^\sharp T\geq_A T T^\sharp \geq_A \alpha^2T^\sharp T$$ and
$$\beta^{\prime 2}S^\sharp S \geq_B S S^\sharp \geq_B \alpha^{\prime 2}S^\sharp S.$$ It follows from the inequalities above  and Lemma 2.2 that
$$\beta^2\beta^{\prime 2}T^\sharp T \otimes S^\sharp S \geq_{A\otimes B}T T^\sharp \otimes S S^\sharp \geq_{A\otimes B} \alpha^2\alpha^{\prime 2}T^\sharp T \otimes S^\sharp S $$ and so
$$(\beta\beta ^\prime)^2\big(T\otimes S\big)^\sharp \big(T\otimes S\big)\geq_{A\otimes B}\big(T\otimes S\big)\big(T\otimes S\big)^\sharp \geq_{A\otimes B} (\alpha\alpha^\prime)^2\big(T\otimes S\big)^\sharp \big(T\otimes S\big).$$ Hence,$T\otimes S$ is a $(\alpha\alpha^\prime,\beta\beta^\prime)$-$A\otimes B$-normal operator.\par \vskip 0.2 cm \noindent
Conversely assume that $T\otimes S$ is a $(\alpha,\beta)$-$A\otimes B$-normal operator.\par \vskip 0.2 cm \noindent
We have
$$\beta^2T^\sharp T \otimes S^\sharp S \geq_{A\otimes B}T T^\sharp \otimes S S^\sharp \geq_{A\otimes B} \alpha^2T^\sharp T \otimes S^\sharp S. $$
So
\begin{equation}\beta^2T^\sharp T \otimes S^\sharp S \geq_{A\otimes B}T T^\sharp \otimes S S^\sharp  \end{equation} and
\begin{equation}T T^\sharp \otimes S S^\sharp \geq_{A\otimes B} \alpha^2T^\sharp T \otimes S^\sharp S \end{equation}
We deduce from inequality $(2.2)$ and Proposition 2.6 that there exists a constant $d>0$ such that
$$
\left\{
\begin{array}{lll}
d \beta^2T^\sharp T\geq _A  TT^\sharp\\
 \\
 \hbox{and}\\
 \\
d^{-1}S^\sharp S \geq_B SS^\sharp \\
\end{array}
  \right.$$

  $$d \beta^2\sup_{\left \|u \right\|_A=1}\left\langle T^\sharp Tu\;|\;u \right\rangle_A \geq \sup_{\left \|u \right\|_A=1}\left\langle T T^\sharp u\;|\;u \right\rangle_A $$ and so
  $$d \beta^2 \left \|T^\sharp T \right\|_A \geq \left \|TT^\sharp \right\|_A.$$ Thus, $d\beta^2\geq 1.$ Similarly, we obtain $ d^{-1} \geq 1.$ \par \vskip 0.2 cm \noindent On the other had by inequality $(2.3)$ and we can find a constant $d_0>0$ satisfies

  $$
\left\{
\begin{array}{lll}
d_0 T T^\sharp\geq _A \alpha^2 T^\sharp T\\
 \\
 \hbox{and}\\

 \\
d_0^{-1}S S^\sharp \geq_B S^\sharp S \\
\end{array}
  \right.$$ It easily to see that $$\sqrt{d_0^{-1}}\alpha \leq 1\;\;\hbox{and}\; d_0\leq 1.$$ Consequently we have
  $$\big(\sqrt{d}\beta\big)^2T^\sharp T \geq_A TT^\sharp \geq_A \big(\sqrt{d_0^{-1}}\alpha\big)^2T^\sharp T$$ and
  $$\big(\frac{1}{\sqrt{d}}\big)^2S^\sharp S \geq_B SS^\sharp \geq_B \big(\sqrt{d_0}\big)^2S^\sharp S.$$
  This proof is completes.
\end{proof}

\begin{theorem}
Let $T,S \in \mathcal{B}_A(\mathcal{H})$ such that $T$is $(\alpha,\beta)$-$A$-normal and $S$ is $(\alpha^\prime,\beta^\prime)$-$A$ -normal operators with $0\leq\alpha \leq 1\leq \beta$  and $0\leq\alpha^\prime \leq 1\leq \beta ^\prime $. The following statements hold:\par \vskip 0.2 cm \noindent (1)
If $T^\sharp S=ST^\sharp$, then $TS\otimes T$ is $(\alpha^2\alpha^\prime,\beta^2\beta^\prime)$-$(A\otimes A)$-normal operator and $TS\otimes S$ is

  $(\alpha\alpha^{\prime 2},\beta\beta^{\prime 2})$-$A\otimes A$-normal operator
\par \vskip 0.2 cm \noindent (2)
If $S^\sharp T=TS^\sharp$,then $ST\otimes T$ is $(\alpha^\prime\alpha^2,\beta^\prime\beta^2)$-$(A\otimes A)$-normal operator and $ST\otimes S$ is

  $(\alpha^{\prime2}\alpha,\beta^{\prime 2}\beta)$-$A\otimes A$-normal operator.
\end{theorem}
\begin{proof}
The proof is an immediate consequence of Theorem 2.3 and Proposition 2.7.
\end{proof}
\begin{center}
\section{  INEQUALITIES  INVOLVING $A$-OPERATOR NORMS AND $A$-NUMERICAL RADIUS OF  $(\alpha,\beta)$-$A$-NORMAL OPERATORS  }
\end{center}
Drogomir and Moslehian $[7]$ have given various inequalities between the operator norm  and the numerical radius  of $(\alpha,\beta)$-normal operators in Hilbert spaces. \par \vskip 0.2 cm \noindent Motivated by this work, we will extended some of these inequalities to $A$-operator norm and $A$-numerical radius $\omega_A$ of $(\alpha,\beta)$-$A$-normal
in semi-Hilbertian spaces by employing some known results for vectors in inner product
spaces.
We start with
the following lemma   reproduced  from $[13]$.
\begin{lemma}
Let
$r \in \mathbb{R}$ and $u,v \in \mathcal{H}$ such that  $\left\|u\right\|_A \geq \left\|v\right\|_A$ and $u,v \notin \mathcal{N}(A)$ ,then the following inequalities hold
 \begin{equation}\left\| u\right\|_A^{2r}+\left\| v\right\|_A^{2r}-2\left\|u \right\|_A^{r}\left\| v\right\|_A^{r}\frac{Re\left\langle u\;|\;v\right\rangle_A}{\left\| u\right\|_A\left\| v\right\|_A}\leq
\left\{
\begin{array}{lll}
r^{2}\left\| u\right\|_A^{2r-2}\left\| u-v\right\|_A^{2}\;\;\hbox{if}\;\;r\geq 1\\
 \\
 \hbox{and}\\
 \\
\left\| v\right\|_A^{2r-2}\left\| u-v\right\|_A^{2}\;\;\hbox{if}\;\;r< 1.
\end{array}
  \right.\end{equation}
\end{lemma}

\begin{theorem}
$T\in \mathcal{B}_A(\mathcal{H})$ be an $(\alpha,\beta)$-$A$-normal operator. Then
\begin{equation} \bigg(\alpha^{2r}+\beta^{2r} \bigg)\left\| T\right\|_A^{2}\leq
\left\{
\begin{array}{lll}
2\beta^r\omega_A(T^2)+\beta^{2r-2}\left\|\beta T-T^\sharp \right\|_A^{2}\;\;\hbox{if}\;\;r\geq 1\\
 \\
 \hbox{and}\\
 \\
2\beta^r\omega_A(T^2)+\left\| \beta T-T^\sharp\right\|_A^{2}\;\;\hbox{if}\;\;r< 1.
\end{array}
  \right.\end{equation}
\end{theorem}
\begin{proof}
Firstly , assume that $r\geq 1$ and let $u \in \mathcal{H}$ with $\left \| u \right\|_A=1$. Since $T$ is $(\alpha,\beta)$-$A$-normal
$$\alpha^2\left\|Tu\right\|_A^2\leq \left\|T^\sharp u \right\|_A^2\leq \beta^2\left\| Tu\right\|_A^2$$  we have
$$\bigg(\alpha^{2r}+\beta^{2r} \bigg)\left\|Tu\right\|_A^{2r}\leq\beta^{2r}\left\|Tu\right\|_A^{2r}+\left\|T^\sharp u\right\|_A^{2r}.$$
Applying Lemma 3.1 with the choices $u_0= \beta\ Tu$ and
$v_0=T^\sharp u$ we get
\begin{equation}
\left\|\beta Tu\right\|_A^{2r}+\left\|T^\sharp u\right\|_A^{2r}-2\left\|\beta Tu\right\|_A^{r-1}\left\|T^\sharp u\right\|_A^{r-1}Re\left\langle \beta Tu\;|\;T^\sharp u\right\rangle_A\leq r^2\left\|\beta Tu\right\|_A^{2r-2}\left\|\beta Tu-T^\sharp u\right\|_A^2.
\end{equation}
From which , it follows that
\begin{eqnarray}
\bigg(\alpha^{2r}+\beta^{2r} \bigg)\left\|Tu\right\|_A^{2r}&\leq&
2\left\|\beta Tu\right\|_A^{r-1}\left\|T^\sharp u\right\|_A^{r-1}\left|\left\langle \beta T^2u\;|\; u\right\rangle_A\right|\nonumber\\&&+ r^2\left\|\beta Tu\right\|_A^{2r-2}\left\|\beta Tu-T^\sharp u\right\|_A^2.
\end{eqnarray}
Taking the supremum in $(3.4)$ over  $u \in  \mathcal{H}, \|u\|_A = 1$  and using the fact that

$$\sup_{\left\| u\right\|_A=1}\left|\left\langle T^2u\;|\;u\right\rangle\rangle_A\right|=\omega_A(T^2)$$ we get
$$\bigg( \alpha^{2r}+\beta^{2r}\bigg)\left\|T\right\|_A^{2r}\leq 2\beta^{r}\left\|T\right\|_A^{2r-2}\omega_A(T^2)+r^2\beta^{2r-2}\left\|T\right\|_A^{2r-2}
\left\|\beta T-T^\sharp \right\|_A^2 .$$ So
$$\bigg( \alpha^{2r}+\beta^{2r}\bigg)\left\|T\right\|_A^{2}\leq 2\beta^{2r}\omega_A(T^2)+r^2\beta^{2r-2}
\left\|\beta T-T^\sharp \right\|_A^2 $$  which is the first inequality in $(3.2).$\par \vskip 0.2 cm \noindent By employing a similar argument to that used in the first inequality  in $(3.1)$ , gives the  second inequality of $(3.2).$

\end{proof}
\begin{theorem}
Let $T\in \mathcal{B}_A(\mathcal{H})$ be an $A(\alpha.\beta)$-normal operator. Then
\begin{equation}\omega_A(T)^2\leq \frac{1}{2}\bigg(\beta \left\|T\right\|_A^2+\omega_A(T^2) \bigg).\end{equation}
\end{theorem}
\begin{proof}
Since for all $u,v $ and $e \in \mathcal{H}$
$$\left|\left\langle u|\;v\right\rangle_A-\left\langle u|\;e \right\rangle_A\left\langle e|\;v\right\rangle_A\right|\geq\left| \left\langle u|\;e \right\rangle_A\left\langle e|\;v\right\rangle_A \right|-\left|\left\langle u|\;v \right\rangle_A\right| $$

we have by applying the inequalities reproduced  from
$[11]$
$$\left\| u\right\|_A\left\| v\right\|_A\geq \left|\left\langle u|\;v\right\rangle_A- \left\langle u|\;e \right\rangle_A\left\langle e|\;v\right\rangle_A\right|+\left| \left\langle u|\;e \right\rangle_A\left\langle e|\;v\right\rangle_A\right|\geq \left|\left\langle u|\;v\right\rangle_A \right|$$ that
\begin{equation}
\left|\left\langle u\;|\;e\right\rangle_A  \right|\left|\left\langle e\;|\;v\right\rangle_A  \right|\leq \frac{1}{2}\bigg(\left\| u\right\|_A  \left\| v\right\|_A  +\left|\left\langle u\;|\;v\right\rangle_A  \right|
\bigg)
\end{equation}
for all $u,v,e  \in \mathcal{H}$ with $\left\|e \right\|_A=1.$\par \vskip 0.2cm \noindent
Let  $x  \in \mathcal{H}$ with $\left\|x \right\|_A=1$ and choosing in $(3.6)$  $u=Tx$ , $v=T^\sharp x$ and $e=x$  we get

\begin{equation}
\left|\left\langle Tx\;|\;x\right\rangle_A  \right|\left|\left\langle x\;|\;T^\sharp x\right\rangle_A  \right|\leq \frac{1}{2}\bigg(\left\| Tx\right\|_A  \left\| T^\sharp x\right\|_A  +\left|\left\langle Tx\;|\;T^\sharp x\right\rangle_A  \right|
\bigg).
\end{equation}
Since $T$ is $(\alpha,\beta)$-$A$-normal, it follows that
\begin{equation} \left|\left\langle Tx\;|\;x\right\rangle_A \right|^2\leq \frac{1}{2} \bigg(\beta \left\| Tx\right\|_A^2+\left|\left\langle T^2x\;|\;x\right\rangle_A\right| \bigg).
\end{equation}
Tanking the supremum over $x\in \mathcal{H}$ $\left\|x\right\|_A=1,$ we get the desired inequality in $(3.5).$
\end{proof}
\begin{theorem}
Let $T\in \mathcal{B}_A(\mathcal{H})$ be an $(\alpha,\beta)$-$A$-normal operator and $\lambda \in \mathbb{C}$. Then
\begin{equation} \alpha\left\|T \right\|_A^2\leq \omega_A(T^2)+\frac{2\beta \left\|T-\lambda T^\sharp\right\|_A^2}{\big(1+|\lambda|\alpha \big)^2} .\end{equation}
\end{theorem}
\begin{proof} For $\lambda =0$, the inequality (3.9) is obvious. Assume that $\lambda \not=0$.From the following inequality $[13]$

$$\frac{1}{2}\bigg(\left\| u\right\|+\left\| v\right\|\bigg)\left\|\frac{u}{\left\| u\right\|}-\frac{v}{\left\| v\right\|} \right\|\leq \left\|u-v \right\|,\;;\;u,v\in \mathcal{H}-\{0\} $$ which is well known in the literature as the Dunkl-Williams inequality, it follows that

$$\frac{1}{2}\big(\left\|u \right\|_A +\left\|v \right\|_A\big)\left\|\frac{u}{\left\|u\right\|_A}-\frac{v}{\left\|v \right\|_A} \right\|_A\leq \left\|u-v \right\|_A \;\;\hbox{for all }\;\;u,v \in \mathcal{H}\;/\;u,v \notin \mathcal{N}(A).$$
A simple computation shows that
$$\left\|\frac{u}{\left\|u\right\|_A}-\frac{v}{\left\|v \right\|_A} \right\|_A^2=2-2\frac{Re\left\langle u\;|\;v\right\rangle_A}{\left\|u \right\|_A \left\|v \right\|_A}\leq \frac{4\left\|u -v\right\|_A ^2}{\big(\left\|u \right\|_A +\left\|v \right\|_A \big)^2}$$ which shows that
$$\frac{\left\|u \right\|_A \left\|v \right\|_A-\left|\left\langle u\;|\;v\right\rangle_A\right|}{\left\|u \right\|_A \left\|v \right\|_A} \leq\frac{2\left\|u -v\right\|_A ^2}{\big(\left\|u \right\|_A +\left\|v\right\|_A \big)^2}, \;\;\hbox{for all }\;\;u,v \in \mathcal{H}\;/\;u,v \notin \mathcal{N}(A)\;\;$$ and so

$$\left\|u \right\|_A \left\|v \right\|_A\leq \left|\left\langle u\;|\;v\right\rangle_A\right|+ \frac{2\left\|u\right\|_A\left\|v \right\|_A}{\big( \left\|u\right\|_A+\left\|v \right\|_A\big)^2}\left\|u-v\right\|_A^2.$$
 Let $x \in \mathcal{H}$ with $\left\|x \right\|_A=1$ and consider $u=Tx$ and $v=\lambda T^\sharp x$ with  $x \notin \mathcal{N}(A^{\frac{1}{2}}T)= \mathcal{N}(A^{\frac{1}{2}}T^\sharp)$ we obtain
 $$\left\|Tx \right\|_A \left\|\lambda T^\sharp x \right\|_A\leq \left|\left\langle Tx\;|\;\lambda T^\sharp x\right\rangle_A\right|+ \frac{2\left\|Tx\right\|_A\left\|\lambda T^\sharp x \right\|_A}{\big( \left\|Tx\right\|_A+\left\|\lambda T^\sharp x \right\|_A\big)^2}\left\|Tx-\lambda T^\sharp x\right\|_A^2.$$ Since $T$ being $(\alpha,\beta)$-$A$-normal operator, we get
 $$\alpha \left\|Tx\right\|_A^2\leq \left|\left\langle T^2x\;|\; x\right\rangle_A\right|+\frac{2\beta\left\|Tx\right\|_A^2}{\big(\left\|Tx\right\|_A+\alpha|\lambda|\left\|Tx\right\|_A\big)^2}\left\|Tx-\lambda T^\sharp x\right\|_A^2.$$
 Tanking the supremum over $x\in \mathcal{H};$ $\left\|x\right\|_A=1,$ we get the desired inequality in $(3.9).$

\end{proof}
\begin{theorem}

Let $T\in \mathcal{B}_A(\mathcal{H})$ be an $(\alpha,\beta)$-$A$-normal operator and $\lambda \in \mathbb{C}$. Then
\begin{equation}
\bigg[  \alpha^2-\big( \frac{1}{|\lambda|}+\beta\big)^2\bigg]\left\| T\right\|_A^4  \leq \omega_A(T^2).
\end{equation}

\end{theorem}
\begin{proof}
We apply the following inequality inspired from $[{8}]$
\begin{equation}
0\leq \left\| u \right\|_A^2\left\| v \right\|_A^2-\left|\left\langle u\;|\;v  \right\rangle_A \right|\leq \frac{1}{|\lambda|^2}\left\| u\right\|_A^2\left\| u-v \right\|_A^2
\end{equation}
for all $u,v \in \mathcal{H}$  and $\lambda \in \mathbb{C}\;, \lambda\not=0.$\par \vskip 0.2 cm \noindent Let $x \in \mathcal{H}$ and set $u=Tx$ and $v=T^\sharp x$ in $(3.11)$ we get
\begin{equation}
\alpha^2\left\| Tx\right\|_A^4\leq \left| \left\langle T^2x\;|\;x \right\rangle_A \right|^2+\frac{1}{|\lambda|^2}
\left\| Tx\right\|_A^2\big(1+|\lambda| \beta \big)^2
\left\| Tx\right\|_A^2.
\end{equation}
Tanking the supremum over $x\in \mathcal{H}$ $\left\|x\right\|_A=1,$ we get the desired inequality in $(3.10).$
\end{proof}
 
\end{document}